\documentclass[12pt,oneside,spanish]{amsart}
\usepackage[T1]{fontenc}
\usepackage[latin9]{inputenc}
\usepackage[letterpaper]{geometry}
\usepackage{array}
\usepackage{float}
\usepackage{multirow}
\usepackage{amstext}
\usepackage{amsthm}
\usepackage{amssymb}
\usepackage{setspace}
\usepackage[all]{xy}
\makeatletter

\providecommand{\tabularnewline}{\\}

\numberwithin{equation}{section}
  \theoremstyle{definition}
  \newtheorem{defn}{\protect\definitionname}
  \theoremstyle{definition}
  \newtheorem{example}{\protect\examplename}
  \theoremstyle{plain}
  \newtheorem{prop}{\protect\propositionname}
 \ifx\proof\undefined\
   \newenvironment{proof}[1][\proofname]{\par
     \normalfont\topsep6\p@\@plus6\p@\relax
     \trivlist
     \itemindent\parindent
     \item[\hskip\labelsep
           \scshape
       #1]\ignorespaces
   }{%
     \endtrivlist\@endpefalse
   }
   \providecommand{\proofname}{Proof}
 \fi
  \theoremstyle{plain}
  \newtheorem{cor}{\protect\corollaryname}

\usepackage{amssymb,amsmath,amsthm,mathrsfs,keyval,color,multirow,lscape}
\usepackage[all]{xy}
\usepackage{graphics}
\usepackage{rotating} 
\usepackage[spanish,es-tabla]{babel}

\AtBeginDocument{

}

\makeatother

\usepackage{babel}
\addto\shorthandsspanish{\spanishdeactivate{~<>}}

\providecommand{\corollaryname}{Corolario}
\providecommand{\definitionname}{Definición}
\providecommand{\examplename}{Ejemplo}
\providecommand{\propositionname}{Proposición}

\begin{document}

\title{Composición de relaciones y $\tau$-factorizaciones}
\maketitle
\begin{center}
David Méndez\footnote{Docente Universitario, Facultad de Ciencias, UNAH: david.mendez@unah.edu.hn}
\par\end{center}

\begin{center}
\textbf{Resumen}
\par\end{center}

La teoría de $\tau$-factorizaciones en dominios integrales fue desarrollada
por Anderson y Frazier, la misma caracterizó las factorizaciones conocidas
y abrió las puertas para crear otras. Se puede visualizar como una
restricción a la operación de multiplicación de la estructura; considerando
una relación simétrica $\tau$ sobre los elementos no invertibles
y distintos de cero de un dominio integral. 

Este trabajo tiene como objetivo principal estudiar e investigar el
concepto de $\tau$-factorizaciones cuando $\tau$ es la composición
de dos o más relaciones. Para poder trabajar con este concepto, se
verifica qué propiedades en específico se pueden obtener a partir
de las relaciones dadas. Entre estas propiedades se estudió las más
conocidas: reflexividad, simetría, transitividad, antisimetría; y
otras asociadas a la teoría de $\tau$-factorizaciones como las relaciones
divisivas, que preservan asociados y multiplicativas. 

\textit{Palabras clave: anillos conmutativos, teoría de factorización,
composición de relaciones}
\begin{center}
\textbf{Abstract}
\par\end{center}

The theory of $\tau$-factorizations on integral domains was developed
by Anderson and Frazier. This theory characterized all the known factorizations
and opened the opportunity to create new ones. It can be visualized
as a restriction to the structure's multiplicative operation, by considering
a symmetric relation $\tau$ on the set of non-zero non-unit elements
of an integral domain. 

The main goal of this work is to study the $\tau$-factorization concept,
when $\tau$ is a composition of two or more relations. To achieve
this, the specific properties one can obtain from the given relations
are verified and analyzed. Some of the studied properties which are
the most known include: reflexivity, symmetry, transitivity, antisymmetry.
And others related to the $\tau$-factorization theory, like: divisive,
associate-preserving and multiplicative relations.


\section*{Introducción}

\begin{doublespace}
La teoría de $\tau$-factorizaciones en dominios integrales fue desarrollada
por Anderson y Frazier en el 2006, un resumen de este trabajo se hace
en Anderson y Frazier (2011), la misma caracterizó las factorizaciones
conocidas y abrió las puertas para crear otras. De esta manera la
teoría generalizó las factorizaciones en dominios integrales conocidas
y estudiadas en años anteriores. Por ejemplo, de las factorizaciones
en elementos irreducibles surgieron los dominios atómicos y de las
factorizaciones en elementos primales surgieron los dominios de Schreier
(Ortiz, 2008).

Este estudio se puede lograr de dos formas. En la primera se consideran
dos relaciones $\tau_{1}$, $\tau_{2}$ y se analiza que resultados
se pueden obtener sobre la relación $\tau_{1}\circ\tau_{2}$. La segunda
forma se basa en tratar de factorizar una relación. Este documento
se enfocó más en la primera forma, detalla algunos elementos de su
complejidad, además de observar como se comportan sus factores, mediante
muchos ejemplos. Para poder trabajar con este concepto, se verifica
qué propiedades en específico se pueden obtener a partir de las relaciones
dadas. 
\end{doublespace}

\section*{Conceptos Básicos}

Dados $A,\,B\subseteq D^{\#}$, el producto cartesiano de $A$ y $B$
se denota y define por $A\times B=\left\{ \left(a,b\right)\,:\,a\in A,\,\,b\in B\right\} .$
Una relación binaria $R$ de $A$ a $B$, es un subconjunto de $A\times B$.
Al conjunto $A$ se le conoce como \textit{dominio} de $R$ y se denota
por $Dom(R)$, al conjunto $B$ se le conoce como \textit{codominio}
de $R$ y se denota como $Codom(R).$ La \textit{coimagen} de $R$,
se define como $Coim(R)=\left\{ a\in A\,:\,\left(\exists b\in B\right)\left((a,b)\in R\right)\right\} $
y la \textit{imagen} de $R$ se define como $Im(R)=\left\{ b\in B\,:\,\left(\exists a\in A\right)\left((a,b)\in R\right)\right\} .$
Estudiamos los tipos clásicos de relaciones: reflexivas, simétricas,
transitivas, de equivalencia y de orden; todas estas definidas en
la forma usual. La definición estándar de composición de relaciones
es la siguiente.
\begin{defn}
Sean $R_{1},\,R_{2}$ dos relaciones sobre $A$, se define la \textit{composición}
$R_{1}\circ R_{2}$ como la relación dada por $aR_{1}\circ R_{2}b$
si y solo si existe $c\in A$ tal que $aR_{2}c$ y $cR_{1}b$. A $R_{1}$
y $R_{2}$ se les conoce como factores de la relación $R_{1}\circ R_{2}$.
\end{defn}
Note que la composición de relaciones no es conmutativa, $Coim(R_{1}\circ R_{2})\subseteq Coim(R_{2})$
e $Im(R_{1}\circ R_{2})\subseteq Im(R_{1})$. Dada una relación $R$
en $A$, se define la relación \textit{inversa} $R^{-1}$ de $R$
dada por $aR^{-1}b$ si y solo si $bRa$. Observar que de las definiciones
de imagen y coimagen se obtiene que $Coim\left(R\right)=Im\left(R^{-1}\right)$
e $Im\left(R\right)=Coim\left(R^{-1}\right)$. Dado un conjunto $A$
y $S\subseteq A$, se define la diagonal o identidad en $S$ por $id_{S}=\{(a,a)\,:\,a\in S\}.$
Note que $id_{Im(R)}\subseteq R\circ R^{-1}$ y $id_{Coim(R)}\subseteq R^{-1}\circ R$.

Sea $D$ un dominio integral, $U(D)$ el conjunto de elementos invertibles
o unidades de $D$ y $D^{\#}$ el conjunto de elementos distintos
de cero que no son unidades de $D$. Un producto $a=\lambda a_{1}a_{2}\cdots a_{n}$
es llamado una $\tau$-factorización de $a\in D^{\#}$, si se cumple
que $a_{i}\tau a_{j}$ para todo $i\neq j$ y $\lambda\in U(D)$.
A los elementos $a_{i}$ se les llama $\tau$-factores de $a$ y $a$
es llamado un $\tau$-producto de los $a_{i}$. Note que si $\tau=D^{\#}\times D^{\#}$,
las $\tau$-factorizaciones y las factorizaciones usuales en $D$
coinciden. Otro ejemplo de relevancia es cuando $\tau=S\times S$,
donde $S\subset D^{\#}$ es un conjunto de elementos distinguidos
de $D^{\#}$.

Los algebristas se han interesado por estudiar estructuras menos exigentes
que la de dominio de factorización única (UFD, por sus siglas en inglés),
por ejemplo, un dominio $D$ se denomina \textit{atómico}, si todos
sus elementos se pueden expresar como producto finito de elementos
irreducibles. Otras estructuras que han resultado importantes son
las siguientes (ver Anderson y Frazier (2011) para más detalles):

(1) \textit{Dominio de factorización acotada} (BFD, por sus siglas
en inglés), \\
(2) \textit{Dominio con la condición de cadenas ascendentes de ideales
principales} (ACCP, por sus siglas en inglés),\\
(3) \textit{Dominio factorial a mitad} (HFD, por sus siglas en inglés),
\\
(4) \textit{Dominio con elementos con una cantidad finita de divisores
irreducibles} (``idf-domain'', por sus siglas en inglés),\\
(5) \textit{Dominio con finitas factorizaciones} (FFD, por sus siglas
en inglés).

\hspace{1cm}Las conexiones entre estos conceptos fué estudiada por
Anderson, Anderson y Zafrullah (1990) y se pueden resumir en la Figura
\ref{fig:Conexi=0000F3n-entre-tipos}. Los autores no solo demostraron
las implicaciones si no que los conversos no se cumplen.
\begin{figure}[h]
\begin{centering}
$\xymatrix{ & \text{HFD}\ar[dr]\\
\text{UFD}\ar[r]\ar[ur]\ar[dr] & \text{FFD}\ar[r]\ar[d] & \text{BFD}\ar[r] & \text{ACCP}\ar[r] & \text{atómico}\\
 & \text{idf-domain}
}
$ 
\par\end{centering}
\centering{}\caption[Conexión entre tipos de dominios. ]{\label{fig:Conexi=0000F3n-entre-tipos}Conexión entre tipos de dominios,
Anderson et. al. (1990) }
\end{figure}
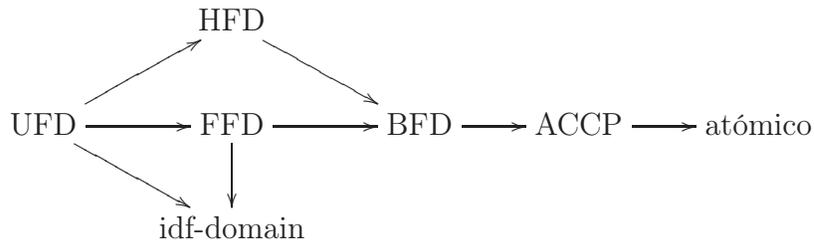

El trabajo de McAdams y Swan (2004) motivó a Anderson y Frazier (2011)
a definir el concepto de $\tau$-factorizaciones, área que llamaron
teoría de factorizaciones generalizadas. La definición de este concepto
fué la siguiente:
\begin{defn}
\label{def:tau-facts}(Anderson y Frazier (2011)) Sea $\tau$ una
relación simétrica sobre $D^{\#}$. Entonces se dice que $a=\lambda a_{1}a_{2}\cdots a_{n}$
es una $\tau$-\textit{factorización} para $a\in D^{\#}$, si $a_{i}\tau a_{j}$
para todo $i\neq j$ y $\lambda\in U(D)$.
\end{defn}
Los autores adaptaron los conceptos de primo, irreducible (o átomo),
UFD, HFD, FFD,...y definieron los respectivos conceptos de $\tau$-primo,
$\tau$-átomo, $\tau$-UFD, $\tau$-HFD, $\tau$-FFD,... además, estudiaron
las conexiones entre estos nuevos tipos de dominios y observaron la
necesidad de crear nuevos tipos de relaciones, que llamaron relaciones
divisivas, que preservan asociados y multiplicativas. En forma resumida,
una relación simétrica $\tau$ es divisiva si cuando $a\tau b$ y
$a'|a$, entonces $a'\tau b$, la relación preserva asociados si cuando
$a\tau b$ y $a'\sim a$, entonces $a'\tau b$ y la relación es multiplicativa
si cuando $a\tau b$ y $a\tau c$, entonces $a\tau bc$. Las conexiones
obtenidas entre estos tipos de dominios se resumen en la Figura \ref{fig:tau-estruc-1}.

\begin{figure}[h]
\begin{centering}
$\xymatrix{\text{UFD}\ar[r]\ar[dd]_{*} & \text{FFD}\ar[r]\ar[d]_{*} & \text{BFD}\ar[r]\ar[dd]_{*} & \text{ACCP}\ar[r]\ar[dd] & \text{atómico}\\
 & \text{\ensuremath{\tau}}\text{-FFD}\ar[dr]\\
\tau\text{-UFD}\ar[ur]\ar[dr] &  & \tau\text{-BFD}\ar[r]_{*} & \tau\text{-ACCP}\ar[r]_{*} & \tau\text{-atómico}\\
 & \tau\text{-HFD}\ar[ur]
}
$ 
\par\end{centering}
\caption[Propiedades de las $\tau$-estructuras]{\label{fig:tau-estruc-1}Propiedades de las $\tau$-estructuras ({*}
significa que $\tau$ es divisiva) (Anderson y Frazier, 2011)}
\end{figure}
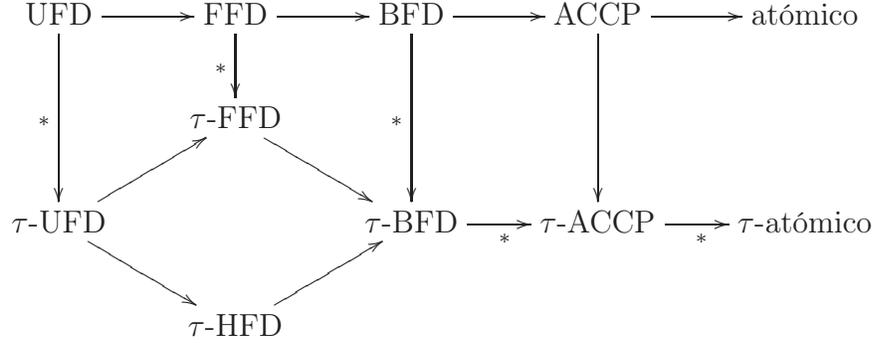

Para propósitos de este trabajo, se redefinieron los conceptos de
divisiva, preservar asociados y multiplicativa, con el objetivo de
que las definiciones sean compatibles con relaciones no necesariamente
simétricas.
\begin{defn}
\label{def:div_laterales} Sean $a,\,a',\,b,\,b',c\in D^{\#}$ y $\tau$
una relación (no necesariamente simétrica) sobre $D^{\#}$.\\
(1) Se dice que $\tau$ es \textit{divisiva por la izquierda} (\textit{derecha}),
si $a\tau b$ y $a'|a$ (resp. $b'|b$), entonces $a'\tau b$ (resp.
$a\tau b'$). Si $\tau$ es divisiva por la izquierda y por la derecha,
 entonces se dice que $\tau$ es \textit{divisiva}. \\
(2) Se dice que $\tau$ \textit{preserva asociados por la izquierda}
(\textit{derecha}), si $a\sim c$ (resp. $b\sim c$) y $a\tau b$,
entonces $c\tau b$ (resp. $a\tau c$). Si $\tau$ preserva asociados
por la izquierda y por la derecha, se dice que $\tau$ \textit{preserva
asociados}. \\
(3) Se dice que $\tau$ es \textit{multiplicativa por la izquierda}
(\textit{derecha}), si $a\tau c$ y $b\tau c$ (resp. $a\tau b$ y
$a\tau c$), entonces $ab\tau c$ (resp. $a\tau bc$). Se dice que
$\tau$ es \textit{multiplicativa}, si es multiplicativa por la izquierda
y por la derecha. 
\end{defn}
Observe que la principal diferencia de estas definiciones con las
originales es que ahora no se requiere que la relación $\tau$ sea
simétrica. A manera de ilustración, considere el siguiente ejemplo.
\begin{example}
Dado un dominio integral $D$, considere la relación $\tau_{\subseteq}$
dada por $a\tau_{\subseteq}b$ si y solo si $(a)\subseteq(b)\subsetneq D$.
Esta relación es claramente reflexiva, transitiva y antisimétrica.
Por lo tanto es un orden parcial, pero no total porque existen ideales
principales no comparables. Por ejemplo, en $\mathbb{Z}$  los ideales
$(p)$ y $(q)$ no son comparables si $p$ y $q$ son primos no asociados.
No es divisiva por la derecha ni por la izquierda. Por ejemplo, en
$\mathbb{Z}[x]$, $x^{6}\tau_{\subseteq}x^{3}$, $x^{2}|x^{3}$ y
$x^{2}|x^{6}$ pero $(x^{2})\not\subseteq(x^{3})$ y $(x^{2})\not\subseteq(x^{6})$,
por lo tanto $(x^{2},x^{3})\notin\tau_{\subseteq}$ y $(x^{2},x^{6})\notin\tau_{\subseteq}$.
Preserva asociados por la izquierda y por la derecha. Dado que si
$a\sim a'$, entonces $(a)=(a')$. Es multiplicativa por la izquierda
pero no por la derecha. Por ejemplo, en $\mathbb{Z}$, $(8)\subseteq(4)$
pero $(8)\not\subseteq(4\cdot4)=(16)$. Es decir, $(8,4)\in\tau_{\subseteq}$,
pero $(8,4\cdot4)\notin\tau_{\subseteq}$.
\end{example}

\section*{Resultados sobre Composiciones}

Como se mencionó en la introducción, se pretende estudiar $\tau$-factorizaciones
cuando $\tau=\tau_{1}\circ\tau_{2}$. Se analiza esta situación estudiando
cuando los factores $\tau_{1}$, $\tau_{2}$ le trasladan propiedades
a la composición $\tau_{1}\circ\tau_{2}$, la razón de esto se puede
observar en el siguiente ejemplo.
\begin{example}
\label{exa:(-p,p)}Suponer que $D$ es UFD y $p\in D$ un elemento
primo. Considerar dos relaciones $\tau_{1}=\left\{ (p,\pm p)\right\} $
y $\tau_{2}=\{(\pm p,p)\}$, entonces $\tau_{1}\circ\tau_{2}=\{(\pm p,\pm p)\}$.
Observe que $\tau_{1}\circ\tau_{2}$ es una relación simétrica, pero
$\tau_{1}$ y $\tau_{2}$ no lo son, $\tau_{1}\circ\tau_{2}$ es divisiva,
pero $\tau_{1}$ y $\tau_{2}$ solo son divisivas por la izquierda
y la derecha, respectivamente. Por otro lado, si $\tau_{1}=\tau_{2}=\{(\pm p,\pm p)\}$,
entonces $\tau_{1}\circ\tau_{2}=\{(\pm p,\pm p)\}$. Esta falta de
unicidad en los factores de la composición hace que este punto de
vista sea menos conveniente.
\end{example}
La composición no se comporta de la manera esperada respecto a los
tipos clásicos de relaciones, la única propiedad que se preserva es
la reflexividad.
\begin{prop}
\label{prop:reflexividad}Si $\tau_{1}$ ($\tau_{2}$) es reflexiva,
entonces $\tau_{2}\subseteq\tau_{1}\circ\tau_{2}$ (resp. $\tau_{1}\subseteq\tau_{1}\circ\tau_{2}$)
. Si ambas son reflexivas, entonces $\tau_{1}\circ\tau_{2}$ y $\tau_{2}\circ\tau_{1}$
también lo son.
\end{prop}
Para los demás casos, se encontraron contraejemplos que muestran que
las demás propiedades no se cumplen y los conversos tampoco. Para
mostrar un caso, consideremos las relaciones de equivalencia, que
generalmente se consideran relaciones que presentan ``buen comportamiento''.
\begin{example}
\label{exa:equiv_dominios_infinitos}Considere en $\mathbb{Z}^{\#}$
las particiones 
\begin{eqnarray*}
\mathcal{P}_{1} & = & (\text{\ensuremath{\mathbb{Z}}}^{-}\backslash\{-1\})\cup\left\{ 2\right\} \cup(\text{\ensuremath{\mathbb{Z}}}^{+}\backslash\{1,2\})\\
\mathcal{P}_{2} & = & (\text{\ensuremath{\mathbb{Z}}}^{-}\backslash\{-1\})\cup\left\{ 2,3\right\} \cup(\text{\ensuremath{\mathbb{Z}}}^{+}\backslash\{1,2,3\}).
\end{eqnarray*}
Estas particiones generan las siguientes dos relaciones de equivalencia
$\tau_{1}$ y $\tau_{2},$ dadas por:
\begin{eqnarray*}
\tau_{1} & = & \left\{ (n_{1},n_{2}),(2,2),(p_{1},p_{2})\,:\,n_{1},n_{2}\in(\text{\ensuremath{\mathbb{Z}}}^{-}\backslash\{-1\}),p_{1},p_{2}\in(\mathbb{Z}^{+}\backslash\{1,2\})\right\} ,\,\,\,\text{y}\\
\tau_{2} & = & \left\{ (m_{1},m_{2}),(2,2),(3,3),(2,3),(3,2),(q_{1},q_{2})\right\} \,\,\,\text{con}\\
 &  & m_{1},m_{2}\in(\text{\ensuremath{\mathbb{Z}}}^{-}\backslash\{-1\})\,\,\,\text{y}\,\,\,q_{1},q_{2}\in(\mathbb{Z}^{+}\backslash\{1,2,3\})
\end{eqnarray*}
Observe que $(2,5)\in\tau_{1}\circ\tau_{2}$, pero $(5,2)\notin\tau_{1}\circ\tau_{2}$;
porque por las definiciones de $\tau_{1}$ y $\tau_{2}$, no existe
un entero $x$ tal que $5\tau_{2}x$ y $x\tau_{1}2$. Por lo tanto,
$\tau_{1}\circ\tau_{2}$ no es una relación de equivalencia.
\end{example}
Para las propiedades relacionadas a $\tau$-factorizaciones, observamos
que ser divisivas y preservar asociados se preservan bajo la composición,
pero la propiedad multiplicativa no lo hace.
\begin{prop}
\label{prop2.4 DIV} Sean $\tau_{1}$ y $\tau_{2}$ relaciones sobre
$D^{\#}$. \\
(1) Si $\tau_{2}$ es divisiva por la izquierda, entonces $\tau_{1}\circ\tau_{2}$
es divisiva por la izquierda.\\
(2) Si $\tau_{1}$ es divisiva por la derecha, entonces $\tau_{1}\circ\tau_{2}$
es divisiva por la derecha.\\
(3) Si $\tau_{1}$ es divisiva por la derecha y $\tau_{2}$ es divisiva
por la izquierda, entonces $\tau_{1}\circ\tau_{2}$ es divisiva. Por
ende, si $\tau_{1}$ y $\tau_{2}$ son divisivas, entonces $\tau_{1}\circ\tau_{2}$
y $\tau_{2}\circ\tau_{1}$ son divisivas.
\end{prop}
\begin{proof}
(1) Sean $a,b,a'\in D^{\#}$ tales que $a'|a$ y $a\tau_{1}\circ\tau_{2}b$.
Por la definición de composición, existe un $c\in D^{\#}$, tal que
$a\tau_{2}c$ y $c\tau_{1}b$. Como $\tau_{2}$ es divisiva por la
izquierda, $a'\tau_{2}c$ y por lo tanto $a'\tau_{1}\circ\tau_{2}b$.
Es decir, $\tau_{1}\circ\tau_{2}$ es divisiva por la izquierda.\\
(2) Si $a,b,b'\in D^{\#}$ son tales que $b'|b$ y $a\tau_{1}\circ\tau_{2}b$.
Por la definición de composición, existe un $c\in D^{\#}$ tal que
$a\tau_{2}c$ y $c\tau_{1}b$ . Como $\tau_{1}$ es divisiva por la
derecha, se tiene que $c\tau_{1}b'$. Por lo tanto, $a\tau_{1}\circ\tau_{2}b'$
y $\tau_{1}\circ\tau_{2}$ es divisiva por la derecha. \\
(3) Esto es consecuencia inmediata de los incisos (1) y (2).
\end{proof}
Los resultados de esta sección se resumen en las Tablas \ref{fig:composicion-divisivas},
\ref{fig:composicion-divisivas-izq} y \ref{fig:composicion-divisivas-der}.
En las casillas centrales se indica si el hecho de que los factores
$\tau_{1}$ y $\tau_{2}$ tengan las propiedades divisiva, divisiva
por la izquierda o divisiva por la derecha, implique que la composición
$\tau_{1}\circ\tau_{2}$ también lo haga.

\begin{table}[H]
\begin{centering}
\caption{\label{fig:composicion-divisivas}Cuando $\tau_{1}\circ\tau_{2}$
es divisiva.}
\smallskip{}
\par\end{centering}
\centering{}%
\begin{tabular}{|c|l|l|l|l|}
\hline 
\multicolumn{1}{|c}{} & \multirow{2}{*}{$\tau_{1}\circ\tau_{2}$} & \multicolumn{3}{c|}{$\tau_{1}$}\tabularnewline
\cline{3-5} 
\multicolumn{1}{|c}{} &  & Divisiva & Div. por la izq. & Div. por la der.\tabularnewline
\hline 
\multirow{3}{*}{$\tau_{2}$} & Divisiva & Divisiva & No divisiva & Divisiva\tabularnewline
\cline{2-5} 
 & Div. por la izq. & Divisiva & No divisiva & Divisiva\tabularnewline
\cline{2-5} 
 & Div. por la der. & No divisiva & No divisiva & No divisiva\tabularnewline
\hline 
\end{tabular}
\end{table}

\begin{table}[H]
\begin{centering}
\caption{\label{fig:composicion-divisivas-izq}Cuando $\tau_{1}\circ\tau_{2}$
es divisiva por la izquierda.}
\par\end{centering}
\centering{}\smallskip{}
\begin{tabular}{|c|l|l|l|l|}
\hline 
\multicolumn{1}{|c}{} & \multirow{2}{*}{$\tau_{1}\circ\tau_{2}$} & \multicolumn{3}{c|}{$\tau_{1}$}\tabularnewline
\cline{3-5} 
\multicolumn{1}{|c}{} &  & Divisiva & Div. por la izq. & Div. por la der.\tabularnewline
\hline 
\multirow{3}{*}{$\tau_{2}$} & Divisiva & Divisiva & Divisiva & No divisiva\tabularnewline
\cline{2-5} 
 & Div. por la izq. & Divisiva & Divisiva & No divisiva\tabularnewline
\cline{2-5} 
 & Div. por la der. & Divisiva & Divisiva & No divisiva\tabularnewline
\hline 
\end{tabular}
\end{table}

\begin{table}[H]
\begin{centering}
\caption{\label{fig:composicion-divisivas-der}Cuando $\tau_{1}\circ\tau_{2}$
es divisiva por la derecha.}
\smallskip{}
\par\end{centering}
\centering{}%
\begin{tabular}{|c|l|l|l|l|}
\hline 
\multicolumn{1}{|c}{} & \multirow{2}{*}{$\tau_{1}\circ\tau_{2}$} & \multicolumn{3}{c|}{$\tau_{1}$}\tabularnewline
\cline{3-5} 
\multicolumn{1}{|c}{} &  & Divisiva & Div. por la izq. & Div. por la der.\tabularnewline
\hline 
\multirow{3}{*}{$\tau_{2}$} & Divisiva & Divisiva & Divisiva & Divisiva\tabularnewline
\cline{2-5} 
 & Div. por la izq. & No divisiva & No divisiva & No divisiva\tabularnewline
\cline{2-5} 
 & Div. por la der. & Divisiva & Divisiva & Divisiva\tabularnewline
\hline 
\end{tabular}
\end{table}

Debido a Anderson y Frazier (2011), se sabe que las relaciones que
son divisivas también preservan asociados, luego en las tablas se
puede sustituir ``Divisiva'' por ``Preserva asociados''. Para
el caso de las relaciones multiplicativas, el siguiente ejemplo muestra
que la composición no lo es, aún cuando ambos factores tengan la propiedad.
\begin{example}
\label{enu:ejemplo-multiplicativas}En $\mathbb{Z}^{\#}$, considere
las siguientes relaciones multiplicativas:
\begin{eqnarray*}
\tau_{1} & = & \left\{ (3^{n},2^{m}),(2^{n},3^{m}),(7^{n},5^{m}),(5^{n},7^{m}):n,m\in\mathbb{Z}^{+}\right\} \\
\tau_{2} & = & \left\{ (3^{n},3^{m}),(3^{n},7^{m}),(7^{n},3^{m}),(3^{n},3^{m}7^{p}),(3^{n}7^{m},3^{p}):n,m,p\in\mathbb{Z}^{+}\right\} .
\end{eqnarray*}
Sus composiciones están dadas por:
\begin{eqnarray*}
\tau_{1}\circ\tau_{2} & = & \left\{ (3^{n},2^{m}),(7^{n},2^{m}),(3^{n}7^{m},2^{p}),(3^{n},5^{m})\,:\,n,m,p\in\mathbb{Z}^{+}\right\} \\
\tau_{2}\circ\tau_{1} & = & \left\{ (2^{n},3^{m}),(2^{n},7^{m}),(2^{n},3^{m}7^{p}),(5^{n},3^{m})\,:\,n,m,p\in\mathbb{Z}^{+}\right\} 
\end{eqnarray*}
Note que para $n,m,p\in\mathbb{\mathbb{Z}}^{+}$, $(3^{n},2^{m}),\,\,(3^{n},5^{m})\in\tau_{1}\circ\tau_{2}$,
pero $(3^{n},2^{m}5^{p})\notin\tau_{1}\circ\tau_{2}$, además $(2^{n},3^{m}),\,\,(5^{n},3^{m})\in\tau_{2}\circ\tau_{1}$,
pero $(2^{n}5^{m},3^{p})\notin\tau_{2}\circ\tau_{1}$. Por tanto,
aunque ambas relaciones sean multiplicativas, la composición no necesariamente
lo es.
\end{example}
Existen varios resultados anteriores que muestran por qué es deseable
trabajar con relaciones multiplicativas, si el objetivo es estudiar
$\tau_{1}\circ\tau_{2}$-factorizaciones, es conveniente saber alguna
forma en la que esta composición es multiplicativa, una manera de
lograrlo es considerando las siguientes propiedades.

\label{Propiedad1-Mult}\textit{Propiedad~(1).} Si $a\tau_{1}\circ\tau_{2}c$
y $b\tau_{1}\circ\tau_{2}c$, entonces existe $d\in D$ tal que $a\tau_{2}d$,
$b\tau_{2}d$ y $d\tau_{1}c$.

\label{Propiedad2-Mult}\textit{Propiedad~(2).} Si $a\tau_{1}\circ\tau_{2}b$
y $a\tau_{1}\circ\tau_{2}c$, entonces existe $d\in D$ tal que $a\tau_{2}d$,
$d\tau_{1}b$ y $d\tau_{1}c$.

Considerando estas dos propiedades, entonces se obtienen los siguientes
resultados.
\begin{prop}
\label{prop:proposicion_composicion_multiplicativa}Sean $\tau_{1}$
y $\tau_{2}$ relaciones sobre $D^{\#}$ tales que $\tau_{1}\circ\tau_{2}\neq\emptyset$.
Entonces,\\
(1) Si $\tau_{2}$ es multiplicativa por la izquierda y tal que cumple
la Propiedad (1), entonces $\tau_{1}\circ\tau_{2}$ es multiplicativa
por la izquierda.\\
(2) Si $\tau_{1}$ es multiplicativa por la derecha y tal que cumple
la Propiedad (2), entonces $\tau_{1}\circ\tau_{2}$ es multiplicativa
por la derecha.
\end{prop}
Como se observará más adelante, existen razones para pensar que en
general, no existen condiciones más débiles en los factores, que hagan
que la composición sea multilplicativa. En la siguiente parte, imponemos
condiciones más fuertes a los factores.

\subsection*{Las condiciones $\tau_{1}\subseteq\tau_{2}$ y $\tau_{1}=\tau_{2}$.}

Una razón importante para considerar este tipo de condiciones es que
en el trabajo de Ortiz {[}\ref{enu:-R.-M.Otiz}{]}, se obtuvo resultados
importantes con condiciones del tipo $\tau_{1}\subseteq\tau_{2}$.
Uno de ellos fué generalizar los resultados de la Figura \ref{fig:tau-estruc-1}.
Ortiz demostró que si $D$ es un $\tau_{2}$-UFD ($\tau_{2}$-BFD,
$\tau_{2}$-FFD y $\tau_{2}$-ACCP) y $\tau_{1}\subseteq\tau_{2}$
dos relaciones divisivas, con $\tau_{2}$ multiplicativa, entonces
$D$ es un $\tau_{1}$-UFD (resp. $\tau_{1}$-BFD, $\tau_{1}$-FFD
y $\tau_{1}$-ACCP). Para el caso de la condición $\tau_{1}=\tau_{2}$,
la asociatividad de la composición nos permite denotarla como $\tau_{1}\circ\tau_{1}=\tau_{1}^{2}$
y en general $\tau_{1}^{n}=\underset{n\,veces}{\underbrace{\tau_{1}\circ\dots\circ\tau_{1}}}$,
veamos algunos resultados obtenidos con esta condición.
\begin{prop}
\label{prop:tau2-de-equiv}Sea $\tau$ un relación en $D^{\#}$ tal
que $\tau^{2}\neq\emptyset$.\\
(1) Si $\tau$ es reflexiva, entonces $\tau^{2}$ es reflexiva.\\
(2) Si $\tau$ es simétrica, entonces $id_{Coim(\tau)\cup Im(\tau)}\subseteq\tau^{2}$
y $\tau^{2}$ es simétrica.\\
(3) Si $\tau$ es transitiva, entonces $\tau^{2}\subseteq\tau$ y
$\tau^{2}$ es transitiva.\\
(4) Si $\tau$ es relación de equivalencia, entonces $\tau^{2}$ es
relación de equivalencia.\\
(5) $Si$ $\tau$ es un orden parcial, entonces $\tau^{2}$ es un
orden parcial.
\end{prop}
Se puede observar que se obtienen mejores resultados que en el caso
general. Pero no se obtiene mejoría respecto a la propiedad multiplicativa.
Además, se encontraron contraejemplos que muestran que aún en este
caso, los conversos de las proposiciones son falsos, es decir, la
composición no le traslada propiedades a sus factores. Relajando un
poco la condición a $\tau_{1}\subseteq\tau_{2}$, obtenemos lo siguiente.
\begin{prop}
\label{prop:transitividad_de_ta1circtau2} Sean $\tau_{1}$ y $\tau_{2}$
relaciones sobre $D^{\#}$ tales que $\tau_{1}\subseteq\tau_{2}$
y $\tau_{2}$ es transitiva. Entonces $\tau_{1}\circ\tau_{2}\subseteq\tau_{2}$
y $\tau_{2}\circ\tau_{1}\subseteq\tau_{2}$. Si además $id_{Im(\tau_{1}\circ\tau_{2})}\subseteq\tau_{1}$
($id_{Im(\tau_{2}\circ\tau_{1})}\subseteq\tau_{1}$), entonces $\tau_{1}\circ\tau_{2}$
es transitiva (resp. $\tau_{2}\circ\tau_{1}$ es transitiva). 
\end{prop}
\begin{proof}
Si $a\tau_{1}\circ\tau_{2}b$, por la definición de composición, existe
un $c\in D^{\#}$ tal que $a\tau_{2}c$ y $c\tau_{1}b$. Como $\tau_{1}\subseteq\tau_{2}$,
entonces $c\tau_{2}b$. Como $\tau$ es transitiva, $a\tau_{2}b$,
por lo tanto $\tau_{1}\circ\tau_{2}\subseteq\tau_{2}.$ Si $a\tau_{2}\circ\tau_{1}b$,
por la definición de composición existe un $c\in D^{\#}$ tal que
$a\tau_{1}c$ y $c\tau_{2}b$. Como $\tau_{1}\subseteq\tau_{2}$,
entonces $a\tau_{2}c$. Como $\tau$ es transitiva, $a\tau_{2}b$,
por lo tanto $\tau_{2}\circ\tau_{1}\subseteq\tau_{2}$. \\
Para la segunda parte, si $a\tau_{1}\circ\tau_{2}b$ y $b\tau_{1}\circ\tau_{2}c$
($a\tau_{2}\circ\tau_{1}b$ y $b\tau_{2}\circ\tau_{1}c$), por la
definición de composición existen $c_{1},\,c_{2}\in D$ tales que
$a\tau_{2}c_{1}$, $c_{1}\tau_{1}b$, $b\tau_{2}c_{2}$ y $c_{2}\tau_{1}c$
(resp. $a\tau_{1}c_{1}$, $c_{1}\tau_{2}b$, $b\tau_{1}c_{2}$ y $c_{2}\tau_{2}c$).
Por la hipótesis de que $\tau_{1}\subseteq\tau_{2}$, se tiene que
$c_{1}\tau_{2}b$ y $c_{2}\tau_{2}c$ (resp. $a\tau_{2}c_{1}$ y $b\tau_{2}c_{2}$).
Como $\tau_{2}$ es transitiva, $a\tau_{2}b$ y $b\tau_{2}c$, implica
$a\tau_{2}c$. Como $id_{Im(\tau_{1}\circ\tau_{2})}\subseteq\tau_{1}$
(resp. $id_{Im(\tau_{2}\circ\tau_{1})}\subseteq\tau_{1}$) y $c\in Im(\tau_{1}\circ\tau_{2})$
(resp. $c\in Im(\tau_{2}\circ\tau_{1})$) , $c\tau_{1}c$. Por lo
tanto $a\tau_{1}\circ\tau_{2}c$ (resp. $a\tau_{2}\circ\tau_{1}c$)
y así $\tau_{1}\circ\tau_{2}$ (resp. $\tau_{2}\circ\tau_{1}$) es
transitiva.
\end{proof}
Las propiedades de ser relación de equivalencia y orden parcial también
observan un mejor comportamiento.
\begin{prop}
Sean $\tau_{1}$ y $\tau_{2}$ relaciones sobre $D^{\#}$, tales que
$\tau_{1}$ es reflexiva, $\tau_{1}\circ\tau_{2}\neq\emptyset$, $\tau_{2}\circ\tau_{1}\neq\emptyset$
y $\tau_{1}\subseteq\tau_{2}$.\\
(1) Si $\tau_{2}$ es relación de equivalencia, entonces $\tau_{1}\circ\tau_{2}$
y $\tau_{2}\circ\tau_{1}$ también lo son. \\
(2) Si $\tau_{2}$ es un órden parcial, entonces $\tau_{1}\circ\tau_{2}$
y $\tau_{2}\circ\tau_{1}$ también lo son.
\end{prop}
Se observa entonces que se obtienen mejores resultados respecto a
algunas de las propiedades clásicas de relaciones. Lastimósamente
esto no ocurre aún con la propiedad multiplicativa, como se muestra
en el siguiente ejemplo.
\begin{example}
Si $\tau_{1}\subseteq\tau_{2}$ y $\tau_{1}$ es multiplicativa por
la izquierda, no necesariamente la composición también lo es. Si se
considera $\tau_{1}=\left\{ (2^{n},2^{m})\,:\,n,m\in\mathbb{Z}^{+}\right\} $
(una relación multiplicativa) y $\tau_{2}=\tau_{1}\cup\left\{ (3,2^{n})\,:\,n\in\mathbb{Z}^{\text{+}}\right\} $.
Entonces $\tau_{1}\circ\tau_{2}=\tau_{2}$. Note que $\tau_{2}$ no
es multiplicativa por la izquierda porque $(2,2),\,(3,2)\in\tau_{1}\circ\tau_{2}$,
pero $(6,2)\notin\tau_{1}\circ\tau_{2}$. Si $\tau_{1}\subseteq\tau_{2}$
y $\tau_{2}$ es multiplicativa por derecha, no se tiene que $\tau_{1}\circ\tau_{2}$
ni $\tau_{1}$ sean multiplicativas por derecha. Considere $\tau_{1}=\left\{ (2,2)\right\} $
y $\tau_{2}=\left\{ (2,2^{n})\,:\,n\in\mathbb{Z}^{+}\right\} $. Entonces
$\tau_{1}\circ\tau_{2}=\tau_{1}$, la cual no es multiplicativa.
\end{example}
Se concluye que aún imponiendo las condiciones $\tau_{1}\subseteq\tau_{2}$
y $\tau_{1}=\tau_{2}$, que se pueden considerar ``fuertes'', estas
no logran que se preserve la propiedad multiplicativa en la composición,
esto justifica la afirmación previa sobre que no haya alguna condición
más débil que la presentada anteriormente. En la siguiente parte se
muestran algunos ejemplos y propiedades de $\tau_{1}\circ\tau_{2}$-factorizaciones
para situaciones particulares, estos casos han sido estudiados anteriormente
y considerado importantes por los autores referenciados en este trabajo.

\section*{algunos ejemplos concretos}

\subsection*{La relación $\tau_{\left(n\right)}$ donde $n\in\mathbb{N}$. }

Sea $D=\mathbb{Z}$ y $n$ un entero positivo fijo, entonces se define
la relación $\tau_{(n)}$ sobre $\mathbb{Z}^{\#}$ como $a\tau_{(n)}b$
si y solo si $a-b\in(n)$. Observe que $a-b\in(n)$ si y solo si $a-b=nk$
para algún $k\in\mathbb{Z}$. Pero esto es equivalente a decir que
$a\equiv b\text{ }(\text{mod }n)$. Es decir, $\tau_{(n)}=\,\,\left(\equiv_{n}\cap\tau_{\mathbb{Z}^{\#}}\right)$,
donde $\equiv_{n}$ es la relación de congruencia módulo $n$ sobre
$\mathbb{Z}$. Por Anderson y Frazier (2011) y Hamon (2007), se conoce
que $\tau_{(n)}$ preserva asociados y es multiplicativa solo cuando
$n=2$; pero nunca es divisiva, si $n>1$. Como $\tau_{(n)}=\,\,\left(\equiv_{n}\cap\tau_{\mathbb{Z}^{\#}}\right)$,
la intersección de dos relaciones de equivalencia sobre $\mathbb{Z}^{\#}$,
$\tau_{(n)}$ también es una relación de equivalencia. Observe que
como $\tau_{(n)}$ una relación simétrica y transitiva, las $\tau$-
factorizaciones coinciden con las $T$-factorizaciones.

Observe que usualmente la relación módulo $n$ en $\mathbb{Z}$, está
definida para $n>1$. Pero la relación $\tau_{(n)}$ se puede definir
para $n\in\mathbb{Z}$. Como $(-n)=(n)$, $\tau_{(-n)}=\tau_{(n)}$.
Por lo tanto, solo se considera cuando $n\geq0$. Si $n=0$, entonces
$\tau_{(n)}=\tau_{(0)}=id_{\mathbb{Z}^{\#}}$, pues dos elementos
se relacionan si y solo si son iguales. Si ambos $n=m=0$, entonces
$\text{gcd}(0,0)$ no está definido. Pero $\tau_{(0)}\circ\tau_{(0)}=\tau_{(0)}=id_{\mathbb{Z}^{\#}}$.
Si $n\neq0$ y $m=0$ , entonces $\tau_{(n)}\circ\tau_{(m)}=\tau_{(n)}$,
pues $\tau_{(0)}=id_{\mathbb{Z}^{\#}}$. Por otro lado, note que $\text{gcd(}n,0)=n$
y $\tau_{(n)}\circ\tau_{(0)}=\tau_{(n)}=\tau_{(\text{gcd}(n,0))}$.
Ahora, suponer que $n,\,m\in\mathbb{Z}^{*}$, por la definición de
composición se tiene que $a\tau_{(n)}\circ\tau_{(m)}b$ si y solo
existe $c\in\mathbb{Z}^{\#}$ tal que $a\tau_{(m)}c$ y $c\tau_{(n)}b$,
es decir que $m|c-a$ y $n|b-c$. Si $n=1$, entonces $\tau_{(1)}=\tau_{\mathbb{Z}^{\#}}$,
pues la diferencia de cualquier dos enteros es divisible por $1$.
La siguiente proposición provee la caracterización de esta composición,
cuando $n$ y $m$ son enteros mayores que $1$.
\begin{prop}
\label{prop:tau_n_circ_tau_m}Si $n,m>1$, entonces $\tau_{(n)}\circ\tau_{(m)}=\tau_{(gcd(m,n))}$.
\end{prop}
\begin{proof}
$\left(\subseteq\right)$ Si $a\tau_{(n)}\circ\tau_{(m)}b$, por la
definición de composición, existe $c\in\mathbb{Z}^{+}$ tal que $m|c-a$
y $n|b-c$. Si $g=gcd(m,n)$, entonces $g|c-a$ y $g|b-c$. Por lo
tanto, $g|(c-a)+(b-c)=b-a$ y $a\tau_{(g)}b$. 

$\left(\supseteq\right)$ Para la otra contenencia, suponer que $g=\text{gcd}(m,n)$
y $a\tau_{(g)}b$. Entonces $g|a-b$ (ó $g|b-a$). Por ende, $gt=a-b$
para algún entero $t$. Por la Identidad de Bezout, existen enteros
$k_{1},\,k_{2}$ tales que $g=mk_{1}+nk_{2}$. Si $n_{1}=tk_{1}$
y $n_{2}=tk_{2}$, entonces $a-b=gt=tmk_{1}+tnk_{2}=mn_{1}+nn_{2}$.
Considere $c=a-mn_{1}=b+nn_{2}$. Despejando se obtiene que $a-c=mn_{1}$
y $c-b=nn_{2}$. Esto quiere decir que $m|a-c$ y $n|c-b$. Por la
definición, se tiene que $c\tau_{(n)}b$ y $a\tau_{(m)}c$. Por la
definición de composición, $a\tau_{(n)}\circ\tau_{(m)}b$.
\end{proof}
\begin{cor}
Sean $m,\,n\in\mathbb{Z}^{+}$. Si $n|m$, entonces 
\end{cor}
(1) $\tau_{(m)}\subseteq\tau_{(n)}$,

(2) $\tau_{(m)}\circ\tau_{(n)}=\tau_{(n)}$, y

(3) $\tau_{(lcm(m,n))}\subseteq\tau_{(m)}\circ\tau_{(n)}$.

Note que estos resultados proveen formas de factorizar la relación
$\tau_{(n)}$ como composición de otras dos, de modo que al menos
para esta relación, se puede predecir qué propiedades (si las hay)
le traslada la composición a sus factores.

\subsection*{La relación $|_{\tau}$.\label{sec:La-relacion_tau_divide}}

Ortiz (2008) desarrolló la relación (que llamó operador) $|_{\tau}$,
que fué definida en Anderson y Frazier (2011) como: dada una relación
simétrica $\tau$ en $D^{\#}$, $a|_{\tau}b$ si existe una $\tau$-factorización
$b=\lambda ab_{1}\cdots b_{n}$ para $b$, donde $a$ aparece como
$\tau$-factor. La expresión ``$a|_{\tau}b$'', se lee ``$a$ $\tau$-\textit{divide}
a $b$''.
\begin{prop}
\label{prop:tau-divide}Sean $\tau_{1}$ y $\tau_{2}$ dos relaciones
sobre $D^{\#}$. Suponer que $a,\,b\in D^{\#}$, \\
(1) Si $a|_{\tau_{1}}b$ y $id_{Coim(\tau_{1})}\subseteq\tau_{2}$,
entonces $a|_{\tau_{1}\circ\tau_{2}}b$.\\
(2) Si $\tau_{1}$ es transitiva,

\hspace{1cm}(a) $\tau_{1}^{2}\subseteq\tau_{1}$, 

\hspace{1cm}(b) las $\tau_{1}^{2}$-factorizaciones son $\tau_{1}$-factorizaciones, 

\hspace{1cm}(c) si $a|_{\tau_{1}^{2}}b$, entonces $a|_{\tau_{1}}b$,
y

\hspace{1cm}(d) los $\tau_{1}$-primos son $\tau_{1}^{2}$-primos.

\end{prop}
\begin{proof}
(1) Si $a|_{\tau_{1}}b$, existe una $\tau_{1}$-factorización $b=\lambda aa_{1}\cdots a_{n}$,
luego $a\tau_{1}a_{1}$ y $a_{i}\tau_{1}a_{i+1}$, para $i\in\{1,...,n-1\}$.
Como $id_{Coim(\tau_{1})}\subseteq\tau_{2}$, $a\tau_{2}a$ y $a_{i}\tau_{2}a_{i}$
para $i\in\{1,...,n\}$. Por la definición de composición, $a\tau_{1}\circ\tau_{2}a_{1}$
y $a_{i}\tau_{1}\circ\tau_{2}a_{i+1}$, para $i\in\left\{ 1,...,n-1\right\} $.
Por lo tanto, $b=\lambda aa_{1}\cdots a_{n}$ también es una $\tau_{1}\circ\tau_{2}$-factorización.\\
(2) Se omiten los detalles.
\end{proof}
Recuerde que $|_{\tau}$ es una relación, luego se puede pensar en
la composición $|_{\tau}\circ|_{\tau}=|_{\tau}^{2}$. Esta nueva relación
no es vacía puesto que $|_{\tau}$ es reflexiva, además se tienen
las siguiente propiedades.
\begin{prop}
Dada una relación $\tau$ sobre $D^{\#}$. \\
(1) Si $\tau$ es divisiva, $|_{\tau}=|_{\tau}^{2}$.\\
(2) Si $\tau$ es transitiva, $|_{\tau^{2}}\subseteq|_{\tau}^{2}$.\\
(3) Si $\tau$ es reflexiva y transitiva, $|_{\tau}\subseteq|_{\tau^{2}}\subseteq|_{\tau}^{2}$.
\end{prop}
Este listado de composiciones y contenencias pueden servir de ejemplos
o contraejemplos en estudios futuros relacionados con los conceptos
de $\tau$-factorizaciones y composiciones. Además, esta relación
brindó la idea de que se pueden caracterizar propiedades de relaciones
en términos de composiciones.

\section*{Trabajos futuros}

\subsection*{Conclusiones}

Este estudio abre el camino para analizar con detalle las $\tau_{1}\circ\tau_{2}$-estructuras.
Se observó las propiedades que se heredan entre $\tau_{1}$, $\tau_{2}$
y su composición $\tau_{1}\circ\tau_{2}$. Se encontró que la composición
de una relación $R$ con una subrelación $S$ de $R$ presenta mejor
comportamiento en heredar propiedades. Se debe indicar, que el comportamiento
de la herencia entre relaciones $\tau_{1}$, $\tau_{2}$ y su composición
$\tau_{1}\circ\tau_{2}$ (ó $\tau_{2}\circ\tau_{1}$), no es el mejor.

\subsection*{Trabajos futuros}

\subsection*{$\tau_{1}\circ\tau_{2}$-estructuras}

Considere en $\mathbb{Z}$ las relaciones $\tau_{1}=\mathbb{Z}^{\#}\times\mathbb{Z}^{\#}$
y $\tau_{2}=\left\{ (6,6),(4,4),(9,9)\right\} ,$ entonces $\tau_{1}\circ\tau_{2}=\left\{ (4,n),(6,n),(9,n)\,\,:\,\,n\in\mathbb{Z}^{\#}\right\} .$
Se observa que $36=6\cdot6$ y ésta es una $\tau_{2}$-factorización
única, pero $36=4\cdot9=6\cdot6$ son dos $\tau_{1}\circ\tau_{2}$-factorizaciones
diferentes. Lo cual implica que el hecho de que $\mathbb{Z}$ sea
un $\tau_{1}$-UFD y un $\tau_{2}$-UFD (las únicas $\tau_{2}$-factorizaciones
no triviales son $4^{n},6^{n}$ y $9^{n}$), no implican que sea un
$\tau_{1}\circ\tau_{2}$-UFD. Esto motiva a preguntarse qué propiedades
deben tener dos relaciones $\tau_{1}$ y $\tau_{2}$ sobre $D^{\#}$
para que: ``Si $D$ es un $\tau_{1}$-UFD y $\tau_{2}$-UFD, entonces
$D$ es un $\tau_{1}\circ\tau_{2}$-UFD''. De igual manera se podría
obtener el diagrama de la Figura \ref{fig:tau-estruc-1}. Claro está
que si $\tau_{1}\circ\tau_{2}$ es divisiva, simétrica y transitiva
el diagrama se satisface, porque las $\tau_{1}\circ\tau_{2}$-factorizaciones
coinciden con el concepto de Anderson y Frazier. Por ende, si $D$
es un UFD, entonces $D$ es un $\tau_{1}\circ\tau_{2}$-UFD. Pero
la idea es reconocer este comportamiento sin asumir que $\tau_{1}\circ\tau_{2}$
ser simétrica y transitiva. 

\subsection*{Composición con homomorfismos}

Sea $\tau$ una relación sobre $D^{\#}$ y $f:D\rightarrow D$ un
homomorfismo de anillos. Analizar una composición de la forma $\tau\circ f$,
fué lo que inicialmente motivó este trabajo. Al examinar muchos ejemplos
se encontró que era necesario primero analizar el comportamiento de
la composición en general. Se pretende a futuro realizar el estudio
de la relación $\tau\circ f$ y su relación con la teoría de $\tau$-factorizaciones.

\section*{Referencias}
\begin{enumerate}
\begin{singlespace}
\item \label{enu:Anderson} D. F. Anderson, D. D. Anderson y M. Zafrullah.
\textit{``Factorization in integral domains''}. J. Pure. Appl. Algebra,
69:1-19,1990.
\item \label{enu:S.-McAdam-and}S. McAdam and R. G. Swan. \textit{``Unique
comaximal factorization''}. J. Algebra, 276(1): 180-192, 2004.
\item \label{enu:-A.-M.Frazier}A. M. Frazier. \textit{``Generalized factorizations
in integral domains''. }Tesis de Doctorado, Universidad de Iowa,
2006. 
\item \label{enu:Hamon} S. M. Hamon. \textit{``Some topics in $\tau$-factorizations''.
}Tesis de Doctorado, Universidad de Iowa, 2007.
\item \label{enu:-R.-M.Otiz} R. M. Ortiz Albino. \textit{``On generalized
nonatomic factorizations'',}Tesis de Doctorado, Universidad de Iowa,
2008. 
\end{singlespace}
\end{enumerate}

\end{document}